\documentclass[12pt]{article}

\usepackage{tabmac}
\usepackage{enumerate, amsmath, amsthm, amsfonts, amssymb, xypic, color,
  mathrsfs, paralist, ulem}
\usepackage[margin=1in]{geometry} 
\usepackage{hyperref}

\input xy
\xyoption{all}

\numberwithin{equation}{section}
\newtheorem{theorem}[equation]{Theorem}

\newtheorem{proposition}[equation]{Proposition}
\newtheorem{lemma}[equation]{Lemma}
\newtheorem{corollary}[equation]{Corollary}

\theoremstyle{definition}
\newtheorem{rmk}[equation]{Remark}

\newtheorem{eg}[equation]{Example}
\newenvironment{example}[1][]{\begin{eg}[#1] \pushQED{\qed}}{\popQED \end{eg}}
\newtheorem{defn}[equation]{Definition}
\newenvironment{definition}[1][]{\begin{defn}[#1]\pushQED{\qed}}{\popQED \end{defn}}

\newcommand{\bC}{\mathbf{C}}

\newcommand{\bP}{\mathbf{P}}

\newcommand{\bQ}{\mathbf{Q}}

\renewcommand{\phi}{\varphi}
\renewcommand{\emptyset}{\varnothing}

\makeatletter
\def\Ddots{\mathinner{\mkern1mu\raise\p@
\vbox{\kern7\p@\hbox{.}}\mkern2mu
\raise4\p@\hbox{.}\mkern2mu\raise7\p@\hbox{.}\mkern1mu}}
\makeatother

\newcommand{\GL}{\mathbf{GL}}

\usepackage{tabmac}

\title{Isotropic Subspaces of Schur Modules}
\author{Leesa B. Anzaldo}
\date{February 5, 2018}
\begin{document} 
\maketitle

\begin{abstract}
It is a well-known fact that over the complex numbers and for a fixed $k$ and $n$, a generic $s$ in $Sym^2V^*$ vanishes on some $k$-dimensional subspace of $V$ if and only if $n\geq 2k$. Tevelev found exact conditions for the extension of this statement for general symmetric and skew-symmetric multilinear forms, and we extend his work to all possible symmetric types, which corresponds to Schur modules for a general partition.
\end{abstract}

\section{Introduction}

Given a generic homogeneous quadratic polynomial over $\bC$, when does its zero set contain a $k$-dimensional subspace? Geometrically, this is equivalent to asking when a generic degree 2 projective hypersurface contains a $(k-1)$-dimensional linear subspace. The answer comes from the well-known fact about symmetric bilinear forms: for a generic $s\in Sym^2V^*$, there exists a $k$-dimensional subspace $W$ of $V$ such that $s|_W=0$ if and only if $n\geq 2k$. One can generalize this question to symmetric multilinear forms: if $V=\bC^n$, when does a generic $s\in (Sym^d V)^*$ vanish on some $k$-dimensional subspace of $V$? Tevelev answers this question for not only symmetric, but also for skew-symmetric multilinear forms in \cite{tv}. Putting aside some exceptions, Tevelev shows that this occurs exactly when
\[n\geq\frac{\binom{d+k-1}{d}}{k}+k\quad\text{or}\quad n\geq\frac{\binom{k}{d}}{k}+k\]
for symmetric or skew-symmetric multilinear forms, respectively.

Here, we consider forms whose symmetries are not considered by Tevelev; this is done by studying the vanishing of $s\in (S_\lambda V)^*$, where $\lambda$ is a nonempty partition which is neither a single row nor a single column partition and $S_{\lambda}$ is the Schur functor associated with $\lambda$. We show that for an $n$-dimensional vector space $V$, partition $\lambda$, and $k\geq3$ such that $2\leq\ell(\lambda)\leq k$ and $\lambda_1\geq2$, a generic $s\in(S_\lambda V)^*$ is $k$-isotropic if and only if
\[n\geq\frac{\dim{(S_\lambda\bC^k)}}{k}+k.\]
Using combinatorial tools, we prove several inequalities about Schur polynomials to obtain this result. Geometrically, we can interpret the main result as precise conditions for the existence of a $k$-dimensional subspace $W$ of $V$ such that $Flag_\lambda(W)$ is in the zero locus of $s$.

\section{Preliminaries}

Let $V=\bC^n$ and $\mathcal{R}$ be the tautological subbundle of $Gr(k,V)$. Recall that for a partition $\lambda$, the Schur module $S_\lambda M$ is a functor with respect to a module $M$, namely the image of the Schur map \cite[76]{fr}. For $s\in H^0(Gr(k,V),S_\lambda\mathcal{R}^*)$, a $k$-dimensional subspace $W$ of $V$ is {\bf isotropic with respect to $s$} if $s(W)=0$. Moreover, we say $s$ is {\bf $k$-isotropic} if there exists a subspace $W$ of $V$ that is isotropic with  respect to $s$. Recall the following theorem (see \cite[Corollary 4.1.9]{we})
\begin{theorem}[Borel-Weil] If $\lambda$ is a partition, then as representations of $\GL(V)$,
\[H^0(Gr(k,V),S_{\lambda}\mathcal{R}^*)=(S_\lambda V)^*.\]
\end{theorem}
Tevelev used the Borel-Weil theorem to generalize the notion of isotropic subspace for symmetric bilinear forms: given $s\in Sym^dV^*$ or $s\in\Lambda^dV^*$, a subspace $W$ of $V$ is {\bf isotropic with respect to $s$} if $s|_{W}=0$. We generalize the definition even further for Schur modules (and this is compatible with the definition for multilinear forms):

\begin{definition} Let $\lambda$ be a partition. For $s\in(S_\lambda V)^*$, a subspace $W$ of $V$ is {\bf isotropic with respect to $s$} if $s|_{S_\lambda W}=0$.
\end{definition}

We answer the following question: for generic $s\in(S_\lambda V)^*$, when does there exist an isotropic subspace $W$ of $V$ with respect to $s$? Tevelev gives necessary and sufficient conditions for the existence of isotropic subspaces with respect to symmetric or skew-symmetric multilinear forms in Theorem~\ref{thm:Tevelev-main}. One could answer this question using the fact that $s\in H^0(Gr(k,V),S_\lambda\mathcal{R}^*)$ is $k$-isotropic if and only if $c_{top}(S_\lambda\mathcal{R}^*)\neq0$, but computing the top Chern class is hard in general:

\begin{example} Let $V=\bC^7$ and $k=5$, and take $s\in\Lambda^3V^*$. By the splitting principle, there exist  line bundles $L_1,\dots,L_5$ from the flag bundle associated with the tautological subbundle of $Gr(5,7)$ such that
\[c(\Lambda^3\mathcal{R}^*)=c\left(\sum_{1\leq i<j<k\leq 5}L_i^{-1}L_j^{-1}L_k^{-1}\right).\]

If the Chern roots of $\mathcal{R}^*$ are denoted by $x_i$'s, then
\[c_{top}(\Lambda^3\mathcal{R}^*)=\prod_{1\leq i<j<k\leq5}\left(x_i+x_j+x_k\right).\]

The Borel presentation of the cohomology ring of $Gr(5,7)$ gives us \[H^*(Gr(5,7))\otimes\bQ=\bQ[x_1,\dots,x_7]^{S_5\times S_2}/I\]
which is the ring of invariant polynomials where $S_5$ acts on $x_1,\dots,x_5$ and $S_2$ acts on $x_6,x_7$, which we then mod out by the ideal $I$ of all positive degree symmetric functions \cite[138]{mn}
Therefore, it is enough to determine whether $c_{top}(\Lambda^3\mathcal{R}^*)$ is in $\left\langle p_1,\dots,p_7\right\rangle$, the ideal generated by the power sum symmetric polynomials in $x_1,\dots,x_7$. This is easily answered using Macaulay2, but difficult by hand:
\begin{verbatim}
QQ[x_1..x_7]
p = k -> sum(apply(7,i->x_(i+1)^k))
f = product(apply(subsets(toList(1..5), 3), s->x_(s_0) + x_(s_1)
 + x_(s_2)));
I = ideal(f)
J = ideal(apply(7, i-> p(i+1)));
isSubset(I,J)
\end{verbatim}

The output is true, and hence the top Chern class is 0, so there does not exist a 5-dimensional isotropic subspace of $V$ with respect to a generic $s$. We arrive at the same conclusion using Tevelev's theorem:
\end{example}

\begin{theorem}[Tevelev] \label{thm:Tevelev-main}
Let $s\in S^dV^*$ or $s\in\Lambda^dV^*$ be a form in general position. The space $V$ contains a $k$-dimensional isotropic subspace with respect to $s$ if and only if
\[n\geq\frac{\binom{d+k-1}{d}}{k}+k\quad\text{or}\quad n\geq\frac{\binom{k}{d}}{k}+k,\quad\text{respectively},\]
with the following exceptions:
\begin{enumerate}
\item if $s\in S^2V^*$ or $s\in\Lambda^2V^*$ is a form in general position, then $V$ contains a $k$-dimensional isotropic subspace if and only if $n\geq 2k$;
\item if $s\in\Lambda^{n-2}V^*$ is in general position and $n$ is even, then $V$ contains a $k$-dimensional isotropic subspace if and only if $k\leq n-2$;
\item if $s\in\Lambda^3V^*$ is in general position and $n=7$, then $V$ contains a $k$-dimensional isotropic subspace if and only if $k\leq4$.
\end{enumerate}
\end{theorem}

\section{Main Theorem}

We give a similar criterion for all partitions $\lambda$ not included in Tevelev's theorem (except for $\lambda = \emptyset$, which is not interesting).

\begin{theorem}\label{thm:main}
Let $V$ be an $n$-dimensional vector space, $\lambda$ be a partition, and take $k\geq3$ such that $2\leq\ell(\lambda)\leq k$ and $\lambda_1\geq2$. Then a generic $s\in(S_\lambda V)^*$ is $k$-isotropic if and only if
\[n\geq\frac{\dim{(S_\lambda\bC^k)}}{k}+k.\]
\end{theorem}
Notice that when rearranged, the inequality can be written as
\[\dim{(Gr(k,n))}\geq\dim{(S_{\lambda}\bC^k)}.\]

\begin{example} Let $V=\bC^6$ and $k=3$, and take $s\in S_{(2,1)}V^*$. By the splitting principle, we can write
\[c(S_\lambda\mathcal{R}^*)=c\left(\sum_T L_1^{-T(1)}L_2^{-T(2)}L_3^{-T(3)}\right)\]
where the $L_i$'s are line bundles, $T$ is a semistandard Young tableau of shape $\lambda$ with entries in $\{1,2,3\}$ (see \ref{def:SSYT} for the definition), and $T(i)$ is the number of boxes in $T$ labeled with $i$. If the Chern roots of $\mathcal{R}^*$ are denoted by $x_i$'s, then we can find out if the top Chern class is 0 by determining whether the product
\[(2x_1+x_2)(2x_1+x_3)(x_1+2x_2)(x_1+x_2+x_3)^2(x_1+2x_3)(2x_2+x_3)(x_2+2x_3)\]
is in $\left\langle p_1,\dots,p_6\right\rangle$, the ideal generated by the power sum symmetric polynomials in $x_1,\dots,x_6$.

This can be answered using Macaulay2:
\begin{verbatim}
QQ[x_1..x_6];
p = k -> sum(apply(6,i->x_(i+1)^k));
f = (2*x_1+x_2)*(2*x_1+x_3)*(x_1+2*x_2)*(x_1+x_2+x_3)^2
     *(x_1+2*x_3)*(2*x_2+x_3)*(x_2+2*x_3);
I = ideal(f);
J = ideal(apply(6, i-> p(i+1)));
isSubset(I,J)
\end{verbatim}

Our final output is false and hence there exists a 3-dimensional isotropic subspace with respect to a generic $s$. This verifies the conclusion of Theorem~\ref{thm:main} since
\[\dim{(Gr(3,6))}=9\geq8=\dim{(S_{(2,1)}\bC^3)}.\qedhere\]
\end{example}

For a geometric interpretation of Theorem~\ref{thm:main}, recall that the more general version of the Borel-Weil theorem \cite[Theorem 4.1.8]{we} says there exists a line bundle $\mathcal{L}(\lambda)$ such that $H^0(Flag_\lambda(V),\mathcal{L}(\lambda))=(S_\lambda V)^*$ as representations of $\GL(V)$. Then the zero locus of $s$ in $(S_\lambda V)^*$, denoted by $Z(s)$, is a subvariety of $Flag_\lambda(V)$. Therefore, we have the following consequence:
\begin{corollary}\label{cor:geometric}
Let $V$ be an $n$-dimensional vector space and $\lambda$ be a partition. For a generic $s\in (S_\lambda V)^*$, there exists a $k$-dimensional subspace $W$ of $V$ such that $Flag_\lambda(W)$ is in the zero locus of $s$ if and only if
\[n\geq\frac{\dim{(S_\lambda\bC^k)}}{k}+k.\]
\end{corollary}

The forward direction of Theorem~\ref{thm:main} is implied by the following general fact, denoted here by Lemma~\ref{lem:sufficient}, which Tevelev also uses in the analogous direction of his proof. In the case of our theorem, $X$ is our Grassmannian and $\mathcal{E}=S_\lambda\mathcal{R}^*$ in the lemma below. Notice $S_\lambda\mathcal{R}^*$ is generated by global sections, i.e. for any $W\in Gr(k,n)$, the map $H^0(Gr(k,n),S_\lambda\mathcal{R}^*)\to (S_\lambda W)^*$, where $s\mapsto s|_{S_\lambda W}$, is surjective. This is true because $H^0(Gr(k,n),S_\lambda\mathcal{R}^*)=(S_\lambda\bC^n)^*$ and $S_\lambda W\hookrightarrow S_\lambda\bC^n$ is injective.

\begin{lemma} \label{lem:sufficient}
Let $X$ be a connected variety of dimension $n$ and $\mathcal{E}$ be a rank $r$ vector bundle on $X$. Assume $\mathcal{E}$ is generated by global sections. If $r>n$, then $Z(s)=\emptyset$ for almost all $s\in H^0(X,\mathcal{E})$.
\end{lemma}

\begin{proof} Define $Z=\{(s,x)\in H^0(X,\mathcal{E})\times X\colon s(x)=0\}$, and let $\pi_1\colon Z\to H^0(X,\mathcal{E})$, $\pi_2\colon Z\to X$ be projection maps. Let $ev_x\colon H^0(X,\mathcal{E})\to\mathcal{E}_x$ take $s\mapsto s(x)$ for $x\in X$. By definition, for any $s\in H^0(X,\mathcal{E})$ and $x\in X$,
\[\pi_1^{-1}(s)\cong Z(s)\]
\[\pi_2^{-1}(x)=\{x\in H^0(X,\mathcal{E})\colon s(x)=0\}=\ker{ev_x}.\]

Since $\mathcal{E}$ is generated by global sections, $ev_x$ is surjective and hence
\[\dim{\pi_2^{-1}(x)}=(\dim{H^0(X,\mathcal{E})})-r.\]

Since $(0,x)\in Z$ for all $x\in X$, $\pi_2$ is a surjective map between irreducible varieties,
\[\dim{Z}=\dim{X}+\max_{x\in X}\{{\dim{\pi_2^{-1}(x)}}\}=n+(\dim{H^0(X,\mathcal{E})})-r\]
so $\dim{Z}<\dim{H^0(X,\mathcal{E})}$.
This implies $\pi_1$ is not surjective, and hence $\overline{\pi_1(Z)}$ is a closed proper subvariety of $H^0(X,\mathcal{E})$. Hence, if $s$ is in the open subset $H^0(X,\mathcal{E})\setminus\overline{\pi_1(Z)}$, then $Z(s)=\pi_1^{-1}(s)=\emptyset$.
\end{proof}

Notice that in the proof of Lemma~\ref{lem:sufficient}, $\pi_2$ is a projective map because it can be factored as $Z\to\bP^n\times X\to X$ where the first map is an isomorphism of $Z$ onto a closed subvariety of $\bP^n\times X$, and the second map is the projection of $\bP^n\times X$ onto $X$. Therefore, $\pi_2(Z)$ is closed, so we have the following corollary:
\begin{corollary}\label{cor:removegeneric} Under the same assumptions as Theorem~\ref{thm:main}, if $n\geq\frac{\dim{(S_\lambda\bC^k)}}{k}+k$, then every $s\in(S_\lambda V)^*$ is $k$-isotropic.
\end{corollary}

To prove the reverse direction of Theorem~\ref{thm:main}, we use a general version of a lemma by Tevelev \cite[p.849]{tv}.

\begin{lemma}[Tevelev] \label{lem:Tevelev-hard} Let $V$ be an $n$-dimensional vector space and $\lambda$ be a partition. If
\begin{equation}
\dim{(S_\lambda\bC^{k-i})}\leq(k-i)(n-k-i)\label{Tevelev-ineq}
\end{equation}
for all $i=0,\dots,\min{\{k,n-k\}}$, then for generic $s\in(S_\lambda V)^*$, $V$ contains a $k$-dimensional isotropic subspace with respect to $s$.
\end{lemma}
In order to show that the inequalities above are satisfied, we compute $\dim{(S_\lambda\bC^{k-i})}$ by evaluating a Schur polynomial $s_\lambda$ in $(1,1,\dots,1)$ and applying tools from combinatorics.
\begin{definition}\label{def:SSYT} For a partition $\lambda$, a {\bf semistandard Young tableau} is a Young diagram of shape $\lambda$ filled with some positive integers so that rows are weakly increasing from left to right and columns are strictly increasing from top to bottom. If $\lambda$ is a partition, then $s_{\lambda}(1^n)$ is the number of semistandard Young tableaux with the shape $\lambda$ and filled with entries from $\{1,\dots,n\}$.

If $b$ is any box in $\lambda$, then the {\bf content of $b$} is $j-i$ if $b$ is in the $i$th row from top to bottom and the $j$th column from left to right; this is denoted by $c(b)$. The {\bf hook length at $b$} is the number of squares below and to the right of $b$, including $b$ once, denoted by $h(b)$.
\end{definition}

\begin{example} These are all possible semistandard Young tableaux of shape $\lambda=(2,1)$ filled with entries from $\{1,2,3\}$:
\[\tableau[s]{1&1\cr 2}\quad\tableau[s]{1&1\cr 3}\quad
\tableau[s]{1&2\cr 2}\quad\tableau[s]{1&2\cr 3}\quad
\tableau[s]{1&3\cr 2}\quad\tableau[s]{1&3\cr 3}\quad
\tableau[s]{2&2\cr 3}\quad\tableau[s]{2&3\cr 3}\quad.\]
Therefore, $s_{(2,1)}(1,1,1)=8$. The values for hook length and content for boxes of $\lambda$ are filled in below:
\[h\colon\tableau[s]{3&1\cr 1}\quad\quad c\colon\tableau[s]{0&1\cr -1}.\qedhere\]
\end{example}

\begin{theorem}[Hook-Content Formula] Let $\lambda$ be a partition and $b$ be any box in $\lambda$. Then
\[s_{\lambda}(1^n)=\prod_{b\in\lambda}\frac{n+c(b)}{h(b)}.\]
\end{theorem}
We use the following well-known result \cite[p.77]{fr}:
\begin{theorem} Let $\lambda$ be a partition. Then
\[\dim{(S_\lambda(\bC^n))}=s_{\lambda}(1^n).\]
\end{theorem}

We can prove that $n\geq\frac{\dim{(S_\lambda\bC^k)}}{k}+k$ implies $\dim{(S_\lambda\bC^{k-i})}\leq(k-i)(n-k-i)$ for most values of $i\in\{0,\dots,\min{\{k,n-k\}}\}$ by showing that
\begin{equation}\label{eq:1}\frac{\dim{(S_\lambda\bC^{k-i})}}{k-i}\geq\frac{\dim{(S_\lambda\bC^{k-i-1})}}{k-i-1}+1.
\end{equation}
Tevelev uses induction to prove \eqref{eq:1}; for example, he gives the following lemma used for $Sym^dV^*$ where $d\geq 3$:

\begin{lemma} \label{lem:Tevelev-symm} If $d\geq 3$ and $\alpha\geq2$, then
\[\frac{\binom{d+\alpha-1}{d}}{\alpha}\geq
\frac{\binom{d+\alpha-2}{d}}{\alpha-1}+1.\]
\end{lemma}

However, this quickly becomes difficult for general partitions. This can be seen in the following examples of hooks and rectangular partitions because $\dim{(S_\lambda\bC^n)}$ is no longer a single binomial coefficient. One can perform a painful induction in particular cases, but it is hard to generalize.
\begin{example} If $\lambda=(d,1)$ where $d\geq2$, then by the Hook-Content Formula,
\begin{align*}
s_{\lambda}(1^n)&=\frac{n}{d+1}\cdot\frac{n+1}{d-1}\cdot\frac{n+2}{d-2}\cdots\frac{n+d-1}{1}\cdot\frac{n-1}{1}\\
&=\frac{(n+d-1)(n+d-2)\cdots(n-1)}{(d+1)(d-1)!}\\
&=\frac{d(n-1)}{d+1}\binom{n+d-1}{d}. \qedhere
\end{align*}
\end{example}

\begin{example} If $\lambda=(d,d)$ where $d\geq2$, then
\[s_\lambda(1^{n})=\frac{n+d-1}{(n-1)(d+1)}\binom{n+d-2}{d}^2.\]

More generally, if $\lambda=(d,\dots,d)$ have $l$ parts where $d,l\geq 2$, then
\[s_\lambda(1^{n})=\left( (l-1)!\binom{n+d-l}{d}\right)^l\prod_{j=1}^{l-1}(n-j)^{j-l}\left(\frac{n+d-j}{j(d+l-j)}\right)^j.\qedhere\]
\end{example}

\section{Proof of Main Theorem}

Now, we give inequalities that will assist in proving \eqref{eq:1}.

\begin{lemma} \label{lem:weak-ineq}
Let $\lambda$ be a nonempty partition.
\begin{enumerate}
\item For $k\geq2$,
\begin{equation}\frac{s_{\lambda}(1^k)}{k}\geq\frac{s_{\lambda}(1^{k-1})}{k-1}. \label{eq:i}
\end{equation}
\item If, in addition, $2\leq\ell(\lambda)\leq k-1$, then
\begin{equation}\frac{s_{\lambda}(1^k)}{k}\geq\frac{s_{\lambda}(1^{k-1})}{k-1}+\frac{1}{k}.\label{eq:ii}
\end{equation}
\end{enumerate}
\end{lemma}
\begin{proof} To prove the second part of the lemma, notice \eqref{eq:ii} is equivalent to
\[k(s_{\lambda}
(1^k)-s_{\lambda}(1^{k-1}))\geq s_{\lambda}(1^k)+k-1.\]

Let $g_{\lambda}(k)$ be the number of semistandard Young tableaux with shape $\lambda$ with entries in $\{1, \dots, k\}$ and labeled with at least one $k$. Since
\[s_{\lambda}(1^k)=g_{\lambda}(k)+s_{\lambda}(1^{k-1}),\]
we can prove the equivalent statement
\[(k-1)g_{\lambda}(k)\geq s_{\lambda}(1^{k-1})+k-1.\]

Let $\mu$ be the subpartition of $\lambda$ obtained by removing the box in the last column of the last row of $\lambda$. Let $\nu$ be the partition obtained by adding a box to the end of the first row of $\mu$. Then
\begin{align*}
(k-1)g_{\lambda}(k)
&\geq(k-1)s_{\mu}(1^{k-1})\\
&=s_{1}(1^{k-1})s_{\mu}(1^{k-1})\\
&\geq s_{\lambda}(1^{k-1})+s_{\nu}(1^{k-1})\\
&\geq s_{\lambda}(1^{k-1})+k-1.
\end{align*}

If we label a partition of shape $\mu$ with entries in $\{1,\dots,k-1\}$, reattach a box to $\mu$ in order to obtain $\lambda$, and label this new box with $k$, then we obtain a semistandard Young tableau of shape $\lambda$ with entries in $\{1,\dots,k\}$; this proves the first line above. The second line is obvious. Since $\ell(\lambda)\geq2$, $\lambda\neq\nu$, so the third line follows by Pieri's rule. Since $\ell(\lambda)\leq k-1$, this implies that $\ell(\nu)\leq k-1$, so $s_\nu(1^{k-1})\neq0$. Moreover, we obtain a semistandard Young tableau if the last box in the first row of $\nu$ is filled with any integer in $\{1,\dots,k-1\}$; the remaining boxes in the first row are labeled with 1; and for the remaining rows, the boxes in the $i$th row are labeled with $i$. Therefore, we have at least $k-1$ semistandard Young tableaux of shape $\nu$ with entries in $\{1,\dots,k-1\}$, proving the fourth line.

Now we prove the first part of the lemma. It is clearly true when $\lambda=(1)$. Otherwise, we again choose $\mu$ to be the subpartition of $\lambda$ obtained by removing the box in the last column of the last row of $\lambda$. Notice that \eqref{eq:ii} is equivalent to
\[(k-1)g_\lambda(k)\geq s_{\lambda}(1^{k-1}).\]

Using a similar reasoning as above, we obtain
\begin{align*}
(k-1)g_{\lambda}(k)
&\geq(k-1)s_{\mu}(1^{k-1})\\
&=s_{1}(1^{k-1})s_{\mu}(1^{k-1})\\
&\geq s_{\lambda}(1^{k-1}).\qedhere
\end{align*}
\end{proof}

\begin{definition} Let $\lambda$ be a partition. If $\mu$ is a subpartition of $\lambda$ such that $\lambda/\mu$ is a skew shape whose columns contain at most one box each, then $\lambda/\mu$ is a {\bf horizontal strip}. We denote the collection of all horizontal strips by $HS$. \end{definition}

The following is a well-known fact:

\begin{proposition} \label{prop:HS} For any partition $\lambda$,
\[s_{\lambda}(1^k)=\sum_{\lambda/\mu\in HS}s_{\mu}(1^{k-1}).\]
\end{proposition}

\begin{proof}
We can partition the collection of all semistandard Young tableau of shape $\lambda$ with entries in $\{1,\dots,k\}$ into subsets based on the placement of $k$'s. Since $k$ can appear at most once in each column of $\lambda$, the size of such a subset is the same as the number of semistandard Young tableau of some unique $\mu\subset\lambda$ such that $\lambda/\mu\in HS$ and labeled with entries in $\{1,\dots,k-1\}$.
\end{proof}

\begin{lemma} \label{lem:strong-ineq}
Let $\lambda$ be a partition and take $k\geq3$. If 
$1\leq\ell(\lambda)\leq k-2$ and $\lambda\neq(1),(2),(1,1)$, then
\[\frac{s_{\lambda}(1^k)}{k}\geq\frac{s_{\lambda}(1^{k-1})}{k-1}+1.\]
\end{lemma}

\begin{proof} We perform induction on $k$. The case when $k=3$ corresponds to symmetric forms, which was proved by Tevelev's lemma \ref{lem:Tevelev-symm}.

Now let $k>3$ and $\lambda$ be a partition satisfying $1\leq\ell(\lambda)\leq k-2$ and not equal to $(1)$, $(2)$, or $(1,1)$. By induction, we suppose that for any partition $\mu$ not equal to $(1)$, $(2)$, or $(1,1)$, and satisfying $1\leq\ell(\mu)\leq k-3$, then
\[\frac{s_{\mu}(1^{k-1})}{k-1}\geq\frac{s_{\mu}(1^{k-2})}{k-2}+1.\]

If $\ell(\lambda)=1$, then we can use Lemma~\ref{lem:Tevelev-symm}. Otherwise, we use Proposition~\ref{prop:HS} several times in the computation below.

\begin{align}
\frac{s_{\lambda}(1^k)}{k}&=\frac{1}{k}\sum_{\lambda/\mu\in HS}s_{\mu}(1^{k-1})\notag\\
&=\frac{1}{k}\sum_{\substack{\lambda/\mu\in HS\\ 1\leq\ell(\mu)\leq k-3\\ \mu\neq(1),(2),(1,1)}}s_{\mu}(1^{k-1})+\frac{1}{k}\sum_{\substack{\lambda/\mu\in HS\\ \ell(\mu)=k-2\\ \mu\neq(1,1)}}s_{\mu}(1^{k-1})+\frac{1}{k}\sum_{\substack{\lambda/\mu\in HS\\ \mu=(1),(2),(1,1)}}s_{\mu}(1^{k-1})\notag\\
&=\frac{k-1}{k}\sum_{\substack{\lambda/\mu\in HS\\ 1\leq\ell(\mu)\leq k-3\\ \mu\neq(1),(2),(1,1)}}\frac{s_{\mu}(1^{k-1})}{k-1}+\frac{k-1}{k}\sum_{\substack{\lambda/\mu\in HS\\ \ell(\mu)=k-2\\ \mu\neq(1,1)}}\frac{s_{\mu}(1^{k-1})}{k-1}\\
&\qquad\qquad +\frac{k-1}{k}\sum_{\substack{\lambda/\mu\in HS\\ \mu=(1),(2),(1,1)}}\frac{s_{\mu}(1^{k-1})}{k-1}\notag\\
&\geq\frac{k-1}{k}\sum_{\substack{\lambda/\mu\in HS\\ 1\leq\ell(\mu)\leq k-3\\ \mu\neq(1),(2),(1,1)}}\left(\frac{s_{\mu}(1^{k-2})}{k-2}+1\right)+\frac{k-1}{k}\sum_{\substack{\lambda/\mu\in HS\\ \ell(\mu)=k-2\\ \mu\neq(1,1)}}\left(\frac{s_{\mu}(1^{k-2})}{k-2}+\frac{1}{k-1}\right)\label{line:1}\\
&\quad+\frac{k-1}{k}\sum_{\substack{\lambda/\mu\in HS\\ \mu=(1),(2),(1,1)}}\frac{s_{\mu}(1^{k-2})}{k-2}\notag\\
&=\frac{k-1}{k(k-2)}\sum_{\lambda/\mu\in HS}s_{\mu}(1^{k-2})+\sum_{\substack{\lambda/\mu\in HS\\ 1\leq\ell(\mu)\leq k-3\\ \mu\neq(1),(2),(1,1)}}\frac{k-1}{k}+\sum_{\substack{\lambda/\mu\in HS\\ \ell(\mu)=k-2\\ \mu\neq(1,1)}}\frac{1}{k}\notag\\
&=\frac{(k-1)^2}{k(k-2)}\frac{s_{\lambda}(1^{k-1})}{k-1}+\sum_{\substack{\lambda/\mu\in HS\\ 1\leq\ell(\mu)\leq k-3\\ \mu\neq(1),(2),(1,1)}}\frac{k-1}{k}+\sum_{\substack{\lambda/\mu\in HS\\ \ell(\mu)=k-2\\ \mu\neq(1,1)}}\frac{1}{k}\notag\\
&\geq\frac{s_{\lambda}(1^{k-1})}{k-1}+\sum_{\substack{\lambda/\mu\in HS\\ 1\leq\ell(\mu)\leq k-3\\ \mu\neq(1),(2),(1,1)}}\frac{k-1}{k}+\sum_{\substack{\lambda/\mu\in HS\\ \ell(\mu)=k-2\\ \mu\neq(1,1)}}\frac{1}{k}\label{line:2}\\
&\geq\frac{s_{\lambda}(1^{k-1})}{k-1}+\frac{k-1}{k}+\frac{1}{k}\label{line:3}\\
&=\frac{s_{\lambda}(1^{k-1})}{k-1}+1.\notag
\end{align}

Line \eqref{line:1} follows by the induction hypothesis on the first sum and Lemma~\ref{lem:weak-ineq} applied to the remaining sums. After rearranging terms and noticing $(k-1)^2>k(k-2)$, we have line~\eqref{line:2}. In Line~\eqref{line:2}, the first sum has at least one summand because we can let $\mu$ be the subpartition of $\lambda$ obtained by removing the last row of $\lambda$; and the second sum has at least one summand because we can take $\mu$ to be $\lambda$. This proves line~\eqref{line:3}.
\end{proof}

We can now prove finish the proof of the main theorem:

\begin{proof}[Proof of Theorem~\ref{thm:main}] Since $s_\lambda(1^j)=0$ for $j<\ell(\lambda)$, by Tevelev's Lemma~\ref{lem:Tevelev-hard} it suffices to show
\[\dim{(S_\lambda\bC^{k-i})}\leq(k-i)(n-k-i)\]
for all $i=0,\dots,k-\ell(\lambda)$.

Suppose $2\leq\ell(\lambda)<k$. Since $i\neq k$, if we assume
\[n\geq\frac{\dim{(S_\lambda\bC^k)}}{k}+k,\]
then using the Lemma~\ref{lem:strong-ineq},
\begin{align*}
n&\geq\frac{s_{\lambda}(1^k)}{k}+k\\
&\geq\frac{s_{\lambda}(1^{k-1})}{k-1}+k+1\\
&\geq\frac{s_{\lambda}(1^{k-2})}{k-2}+k+2\\
&\vdots\\
&\geq\frac{s_{\lambda}(1^{\ell(\lambda)+2})}{\ell(\lambda)+2}+k+k-\ell(\lambda)-2\\
&\geq\frac{s_{\lambda}(1^{\ell(\lambda)+1})}{\ell(\lambda)+1}+k+k-\ell(\lambda)-1.\\
\end{align*}
This proves the inequality~\eqref{Tevelev-ineq} for $i=0,\dots,k-\ell(\lambda)-1$.

If $\lambda$ is a rectangle, then
\[\dim{(S_\lambda\bC^{\ell(\lambda)})}=s_{\lambda}(1^{\ell(\lambda)})=1\leq\ell(\lambda)(n-\ell(\lambda))\]
because $\ell(\lambda)>1$, so the inequality for $i=k-\ell(\lambda)$ is true.

If $\lambda$ is not a rectangle, then let $\mu$ be the partition obtained by removing all columns of height $\ell(\lambda)$ from $\lambda$; therefore, $s_\lambda(1^{\ell(\lambda)}) = s_\mu(1^{\ell(\lambda)})$. If $\mu=(1)$,
then
\[\dim{(S_{\lambda}\bC^{\ell(\lambda)})}=s_{\mu}(1^{\ell(\lambda)})=\ell(\lambda)\leq\ell(\lambda)(n-\ell(\lambda))\]
because $\ell(\lambda)<k\leq n$ . If $\mu=(2)$ or $(1,1)$, then our assumption, $n\geq\frac{\dim{(S_\lambda\bC^k)}}{k}+k$, says
\[n\geq\frac{3k+1}{2},\quad n\geq\frac{3k-1}{2},\]
which imply the desired inequalities
\[n\geq\frac{3\ell(\lambda)+1}{2}=\frac{s_{\lambda}(1^{\ell(\lambda)})}{\ell(\lambda)}+\ell(\lambda),\]
\[n\geq\frac{3\ell(\lambda)-1}{2}=\frac{s_{\lambda}(1^{\ell(\lambda)})}{\ell(\lambda)}+\ell(\lambda),\]
respectively. Otherwise, we can prove \eqref{eq:1} for our remaining case
\begin{align*}
\frac{s_{\lambda}(1^{\ell(\lambda)+1})}{\ell(\lambda)+1}&\geq\frac{s_{\mu}(1^{\ell(\lambda)+1})}{\ell(\lambda)+1}\\
&\geq\frac{s_{\mu}(1^{\ell(\lambda)})}{\ell(\lambda)}+1\\
&=\frac{s_{\lambda}(1^{\ell(\lambda)})}{\ell(\lambda)}+1.
\end{align*}

The first inequality is true because given a semistandard Young tableau of shape $\mu$ filled with entries from $\{1,\dots,\ell(\lambda)+1\}$, one can obtain a semistandard Young tableau of shape $\lambda$ filled with entries from $\{1,\dots,\ell(\lambda)+1\}$ in the following way: adjoin a rectangle to the left of $\mu$ in order to obtain the shape $\lambda$, and label the entire $i$th row of the rectangle with $i$. Since $\ell(\mu)\leq\ell(\lambda)-1$, we can apply Lemma~\ref{lem:strong-ineq} to obtain the second inequality. The last line follows because the rectangle removed from $\lambda$ in order to obtain $\mu$ must be constant along rows when filled with integers $1,\dots,\ell(\lambda)$, and this is done in exactly one way. Therefore, we've proved the case when $i=k-\ell(\lambda)$.

If $\ell(\lambda)=k$, then we need only show \ref{Tevelev-ineq} holds for $i=0$, but this is our assumption on $n$.
\end{proof}

\end{document}